\documentclass[a4paper,11pt]{amsart}
 \usepackage{float}
\usepackage{amssymb}
\usepackage{mathrsfs}
\usepackage{amsmath,amssymb,amsthm,latexsym,amscd,mathrsfs}
\usepackage{indentfirst}
\usepackage{stmaryrd}
\usepackage{graphicx}
\usepackage{extarrows}
\usepackage{multirow}
\usepackage{latexsym}
\usepackage{amsfonts}
\usepackage{color}
\usepackage{pictexwd,dcpic}
\usepackage{graphicx}
\usepackage{psfrag}
\usepackage{hyperref}
\usepackage{comment}
\usepackage{enumerate}

\usepackage{amsmath}
\usepackage{amsfonts}
\usepackage{amssymb}
\usepackage{amsthm}
\usepackage{mathrsfs}
\usepackage{indentfirst}
\usepackage{bbm}
\usepackage{bm}
\numberwithin{equation}{section} \theoremstyle{plain}
\newtheorem{thm}{Theorem}[section]

\newtheorem{lem}[thm]{Lemma}

\newtheorem{defn}[thm]{Definition}

\newtheorem{rem}[thm]{Remark}

 \makeatletter

\def\<{\langle}
\def\>{\rangle}
\def\({\left(}
\def\){\right)}
\def\[{\left[}
\def\]{\right]}

\makeatother
\title{The weighted reproducing kernels of the Reinhardt domain}

\author[Q. Fu]{Qian Fu$^{*}$}
\author[G. T. Deng]{Guantie Deng${ }^{\dagger}$}
\thanks{$^{*}$Laboratory of Mathematics and Complex Systems (Ministry of Education), School of Mathematical Sciences, Beijing Normal University, qianfu@mail.bnu.edu.cn}
\thanks{${ }^{\dagger}$Corresponding author. Laboratory of Mathematics and Complex Systems (Ministry of Education), School of Mathematical Sciences, Beijing Normal University, 96022@bnu.edu.cn}

\thanks{The project is supported by the National Natural Science Foundation of China (Grant no. 12071035 and 11971045).}
\allowdisplaybreaks

\begin{document}
\maketitle

% % % % % % % % % %
\begin{abstract}
In this paper, we develop the theory of weighted Bergman space and obtain a general representation formula of the Bergman kernel function for the spaces on the Reinhardt domain containing the origin. 
As applications, we calculate the concrete forms of the Bergman kernels for some special weights on the Reinhardt domains $\mathbb{C}^n,$ 		
 $D_{n,m}:= \{(z, w)\in \mathbb{C}^n \times \mathbb{C}^m : \|w\|^2 <{e}^{-\mu_1\|z\|^{\mu_2}}
				\}$	 and	$V_{\eta}:=\{\left(z, z', w\right) \in \mathbb{C}^{n} \times \mathbb{C}^{m} \times \mathbb{C} : \sum_{j=1}^{n} e^{\eta_{j}|w|^{2}}\left|z_{j}\right|^{2}+\|z'\|^{2}<1\}$.
			\\
			

%\end{minipage}
%\end{center}
\end{abstract}

\section{Introduction}
Let $\mathcal{D}$ be a domain in $\mathbb{C}^{n}$ and $\mathcal{O}(\mathcal{D})$ be the space of holomorphic functions on $\mathcal{D}$. Denote $L^{2}(\mathcal{D}):=\left\{f: \mathcal{D} \rightarrow \mathbb{C}: \int_{\mathcal{D}}|f(z)|^{2} \mathrm{~d} V(z)<\infty\right\}$, where $\mathrm{d} V(z)$ is the Lebesgue measure. Furthermore, it is known (cf. \cite{SGK1}) that the space $\mathcal{A}^{2}(\mathcal{D}):=\mathcal{O}(\mathcal{D}) \bigcap L^{2}(\mathcal{D})$ 
is a Hilbert space with inner product
$$
\langle f, g\rangle=\int_{\mathcal{D}} f(z) \overline{g(z)} \mathrm{d} V(z),
$$
for $f,g \in \mathcal{A}^2(\mathcal{D})$.
 For each fixed $z \in \mathcal{D}$, the functional
$
 \Phi_z: f \mapsto f(z),  f \in \mathcal{A}^2(\mathcal{D})
 $
 is a continuous linear functional on $\mathcal{A}^2(\mathcal{D})$. Therefore, by the  Riesz representation theorem, there is a unique element in $\mathcal{A}^{2}(\mathcal{D})$, which we denote $K_{\mathcal{D}}(\cdot, z)$, such that
$$
f(z)= \Phi_z(f)=\left\langle f, K_{\mathcal{D}}(\cdot, z)\right\rangle=\int_{\mathcal{D}} f(w) \overline{K_{\mathcal{D}}(w, z)} \mathrm{d} V(w)
$$
for all $f \in \mathcal{A}^{2}(\mathcal{D})$. The function
%$$
%K_{\mathcal{D}}: \mathcal{D} \times \mathcal{D} \longrightarrow \mathbb{C}
%$$
$K(z, w)=K_{\mathcal{D}}(z, w)$ is called the Bergman kernel for $\mathcal{D}$, and $\mathcal{A}^{2}(\mathcal{D})$ is called Bergman space. Let $\left\{\varphi_{k}, k=1,2, \cdots\right\}$ be a complete orthonormal system of $\mathcal{A}^{2}(\mathcal{D}),$ then the kernel function on $\mathcal{D}$ satisfies
$$K_{\mathcal{D}}(z, w)=\sum_{k=1}^{\infty} \varphi_{k}(z) \overline{\varphi_{k}(w)}$$ for all $(z, w) \in \mathcal{D} \times \mathcal{D}.$
%In addition, % (cf. \cite{SGK1,SGK2}, etc). 
For more details, please refer to \cite{Bergman1970,SGK1,SGK2}, etc.

The theory of Bergman space has, in the past several decades, become important in the complex analysis of both one and two complex variables (cf. \cite{SGK1,SGK2}, etc). In this Hilbert setting, the reproducing kernel plays a prominent role, and its reproducing properties and biholomorphic invariance are of fundamental importance.

%Bergman kernels have also been considered in some weighted cases. 
It is important to obtain concrete information about the kernel function. However, we must confess that it is generally hard to obtain concrete representations for the Bergman kernel except for special cases such as the Hermitian ball or polydisc (cf.  \cite{SGK2}, etc). Nevertheless, the explicit formula of the Bergman kernel is used extensively in several areas (cf. \cite{Bergman1970,SGK1,SGK2}, etc), such as in the study of holomorphic invariable metrics, the boundary regularity of biholomorphic maps, and function space theory. 
%An analogous concrete expression was applied to the classical domains of the bounded symmetric domain of Cartan in \cite{HLG}.
%According to the structural features of bounded domains, the Bergman kernel can be computed in many different ways  (cf. \cite{SGK1,SGK2}, etc). 
In 1974 Fefferman \cite{FEF} introduced a new technique for obtaining an asymptotic expansion for the Bergman kernel on a large class of domains. %More precisely, 
%Hence, this method could be used  to compute the Bergman kernel for the unit ball and the unit polydisk in $\mathbb{C}^{n}$.
 D'Angelo
\cite{DA1,DA2} gave the explicit formula of the Bergman kernel function on the domain
$\Omega= \{(z, w) \in \mathbb{C}^{n+m}: ||z||^2 < ||w||^{2p} < 1\},$
for any positive real $p$. Francsics and
Hanges \cite{FG} expressed the Bergman kernel on complex ovals in terms of generalized
hypergeometric functions. Recent interest in the explicit formula of the Bergman kernel is motivated by its surprising applications, please see \cite{AE, GR} on this topic.

Recently, Mai and Shao \cite{MS} studied the Bergman of
generalized Bargmann-Fock spaces in the setting of Clifford algebra. Hezari et al. \cite{HLX} proved a new off-diagonal asymptotic of the Bergman kernels associated to tensor powers of a positive line bundle on a compact K$\ddot{a}$hler manifold. 
Deng et al. \cite{DG2}  obtained the reproducing kernel over  tubular region using Laplace transform, which is an effective and new method of calculating reproducing kernels.
Boas et al. \cite{BHP} introduced a different method. They considered the domain $\Omega=\{(z, w) \in \mathbb{C} \times \mathbb{C}^{n}:|z|<p(w)\},$ where $p(w)$ is a bounded, positive, continuous function on the interior of some bounded domain in $\mathbb{C}^{n} .$ By differentiating the Bergman kernel on $\Omega$, they obtained the kernel function on
$
\left\{(z, w) \in \mathbb{C}^{m} \times \mathbb{C}^{n}:\|z\|<p(w)\right\}.
$
Additional results have been obtained in \cite{BT} on the domain
$
\{\left(z_{1}, z_{2}, z_{3}\right) \in \mathbb{C}^{3}:\left(\left|z_{1}\right|^{2 p}+\left|z_{2}\right|^{4}\right)^{1 / \lambda}+\left|z_{3}\right|^{2 / q}<1\}
$
and in \cite{YA} on the Fock-Bargmann-Hartogs domain
$
\{(z, w) \in \mathbb{C}^{n+m}:\|z\|<e^{-a\|w\|^{2}}\}.
$

In this paper, Bergman kernels will be considered in some weighted cases.  For a positive continuous function $\mu(z)$ on $\mathcal{D}$, we assume that $\mu (z)=0$ for $z\not\in \mathcal{D} $
and  consider the weighted volume measure
\begin{equation*}
	dV_{\mu}(z)=\mu(z)dV(z),
\end{equation*}
where $dV(z)=dxdy $ is the Lebesgue measure on $\mathbb C^n$. For $p>0$, we denote $L^p_{\mu}(\mathcal{D})$ the space of measurable functions in the domains $\mathcal{D} $ such that
\begin{equation*}
	\|f\|_{L^p_{\mu}}=\left(\int_{\mathcal{D}}|f(z)|^pdV_{\mu}(z)\right)^{\frac1p}<\infty.\label{anorm}
\end{equation*}
The space of such functions is called weighted Lebesgue space with weight $\mu$. The quantity $
\|f\|_{L^p_{\mu}} $ is called the norm of the function $f$; it is a true norm if $p\ge 1$.
We also denote $\mathcal{A}^p_\mu(\mathcal{D})$ the subspace of $L^p_{\mu}(\mathcal{D})$, which consists of all holomorphic functions $f$ in the domain $\mathcal{D}$ and $f\in L^p_{\mu}(\mathcal{D})$.
The space of such functions is called weighted Bergman space with weight $\mu$.

In \cite[Lemma 1-2]{DG2}, Deng et al. proved that point evaluation is a bounded linear functional in each weighted Bergman space $\mathcal{A}^2_\mu(\mathcal{D})$ that is closed in $L^2_\mu(\mathcal{D})$. Therefore, $\mathcal{A}^2_\mu(\mathcal{D})$ is a Hilbert space with inner product
%Furthermore, for $p=2$, define $\mathcal{A}^p_\mu(\mathcal{D})$ with inner product
$$
\langle f, g\rangle_\mu=\int_{\mathcal{D}} f(z) \overline{g(z)}   dV_{\mu}(z),
$$
for $F,G \in \mathcal{A}^2_\mu(\mathcal{D})$.
%It is obvious that $\mathcal{A}_{\mu}^{2}(\mathcal{D})$ is a reproducing kernel Hilbert space. 
The Riesz representation theorem for Hilbert space guarantees existence of a unique function $K_z(\cdot)=K_{\mathcal{D}, \mu}(\cdot,z)   \in \mathcal{A}^2_\mu(\mathcal{D})$ such that $F(z) = \langle F,K_z \rangle_\mu$ for every $F \in \mathcal{A}^2_\mu(\mathcal{D})$.
The function $K_{\mathcal{D}, \mu}(z, w)$, $z, w \in \mathcal{D},$ is known as the reproducing kernel with weight $\mu$ for $\mathcal{D}$, or the weighted Bergman kernel function.
When $\mu \equiv 1$, it is just the Bergman kernel $K_{\mathcal{D}}(z, w)$ for $\mathcal{D}$.

%The Bergman kernel can be defined quite easily for each bounded domain, but its explicit formula is very difficult to obtain, except for some special domains, such as the ball or polydisc in $\mathbb{C}^{n}.$ 

Herein, 
we develop the theory of weighted Bergman space and obtain a general representation formula (see Theorem \ref{theorem1}) of the kernel function for the spaces on the Reinhardt domain (see Definition \ref{definition Reinhardt domain}) containing the origin.
As a complement to the overall research, we have computed the specific expressions of the Bergman kernels under certain special weights within the Reinhardt domain, which extend some of the previously reported outcomes,
%{\color{blue}As a complement to the general study, we calculate the concrete forms of theBergman kernels for some special weights on the Reinhardt domain, which are generalizations of some of the results mentioned above,} 
please see   Theorems \ref{example 1}, \ref{example 2} and \ref{example 3} for more details. In Sects.  \ref{lemmas} and \ref{proof},  we give the required lemmas and  prove the theorems,  respectively.

	\section{Main results}\label{results}
%	\section{Applications %of Theorem \ref{theorem1}	 -- Computation of some weighted Bergman kernels}
%We introduce some notations and definitions.
\subsection*{Notations and definitions}
 The product of  $z=(z_1,...,z_n),$  $w=(w_1,...,w_n)\in\mathbb{C}^n$ is $\langle z,w\rangle=z_1\overline{w}_1+z_2\overline{w}_2+\ldots+z_n\overline{w}_n$. The Euclidean norm of $z\in\mathbb{C}^n$  is $\|z\|=\sqrt{\langle z,w\rangle}$.  In any discussion of functions of $n$ variables, the term multi-index refers to an ordered $n$-tuple
$
\alpha=\left(\alpha_1, \ldots, \alpha_n\right)
$
of nonnegative integers $\alpha_i$. The following abbreviated notations will be used:
$
|\alpha|=\alpha_1+\cdots+\alpha_n,
$
$
\alpha !=\alpha_{1} ! \ldots \alpha_{n} !,
$
$
D^{\alpha}=\frac{\partial^{|\alpha|}}{\partial z_1^{\alpha_1} \ldots \partial z_n^{\alpha_n}},
$
$
z^\alpha=z_1^{\alpha_1} \ldots z_n^{\alpha_n},
$
where $z=\left(z_1, \ldots, z_n\right)$. We shall refer to $z^\alpha$ as a holomorphic monomial.

%	Let us now prove Theorem 1.
%	\begin{proof}
%\subsection{}
\begin{defn}\cite{HO}
\label{definition Reinhardt domain}
%A set  $\Omega \subseteq \mathbb{C}^n$ is called 
%circular (or a circled set) if $z \in S$ implies that $\left(e^{i \theta} z_1, \ldots, e^{i \theta} z_n\right) \in S$ for all $0 \leq \theta<2 \pi$. The set is called 
%a Reinhardt domain if  $z =( z_1, z_2,\cdots, z_n)\in\Omega$ implies $\left(e^{i \theta_1} z_1, \ldots, e^{i \theta_n} z_n\right) \in \Omega$ for all $0 \leq \theta_j<2 \pi, j=1, \ldots, n$. The set is called a complete circular domain or complete Reinhardt if $z \in S$ implies that $\left(\mu_1 z_1, \ldots, \mu_n z_n\right) \in \Omega$ for all $\mu_j \in \mathbb{C}$ with $\left|\mu_j\right| \leq 1, j=1, \ldots, n$.
An open set $\Omega\subset \mathbb{C}^n$ called a Reinhardt domain  if  $z = (z_1, \cdots, z_n)\in\Omega$ implies  $(e^{i\theta_1}z_1,\cdots, e^{i\theta_n}z_n)\in \Omega$ for arbitrary real numbers $\theta_1, \theta_2,\cdots,\theta_n$.
\end{defn}

\begin{defn}\cite{HO}
Let $\Omega$ be a Reinhardt domain. Define the Reinhardt shadow of $\Omega$ to be set
$$
\tilde{\Omega} := \{(r_1,r_2,\cdots,r_n) \in \mathbb{R}^n : r_j \geq 0,(r_1,r_2,\cdots,r_n) \in \Omega\}.
$$
Given a point in a Reinhardt domain $z = (z_1, z_2,\cdots,z_n)  \in \Omega$, its point shadow will
be denoted 
$\tilde{z} := (|z_1|,|z_2|,\cdots,|z_n|)\in \tilde{\Omega}.$	 
\end{defn}
%Let $\Omega\subset \mathbb{C}^n$ be a connected Reinhardt domain containing the origin and define it as follows:	$$\Omega =\{z\in \mathbb{C}^n: z=(r_1e^{i\theta_1},r_2e^{i\theta_2},\cdots,r_ne^{i\theta_n}): r\in \tilde{\Omega},\theta=(\theta_1,\theta_2,\cdots,\theta_n)\in \mathbb{R}^n\},$$
%	where $\tilde{\Omega}\subset\mathbb{R}_{+}^n:=\{r=(r_1,\cdots,r_n):r_j\geq0,j=1,2,\cdots,n\}$ is the connected set containing the origin.

%	The main result herein is established as follows.

	\begin{thm}[\bf Main Theorem]
	\label{theorem1}
		 Let $\Omega\subset \mathbb{C}^n$ be a connected Reinhardt domain containing the origin.
		 Suppose that $\varphi (z)$ is a positive continuous function on $\Omega$ and also a radial function about each component $z_j \,(j=1,\cdots, n)$, that is, $\varphi(z)=\varphi(|z_1|,|z_2|,\cdots,|z_n|)$.
%		 $\varphi(z)=\varphi(|z_1|,|z_2|,\cdots,|z_n|)$ is a radial function with respect to each component $z_j \,(j=1,\cdots, n).$
%		 $\varphi (r) $ be a positive continuous function on		 $\tilde{\Omega}$, which is the Reinhardt shadow of $\Omega$.
		 Then the weighted Bergman kernel $K(z,w)$ of $\mathcal{A}^2_\varphi(\Omega)$ 
		 %with weight $\mu(z)=\varphi(|z_1|,|z_2|,\cdots,|z_n|)=\varphi (r)$  
		 is given by
		\begin{equation}
			K(z,w)=\sum_{\alpha\in \mathbb N^n} I^{-1}(\alpha)z^{\alpha}\overline{w}^{\alpha},\label{kernelform}
		\end{equation}
where $I(\alpha)=(2\pi)^n\int_{\tilde{\Omega}}r^{2\alpha+\mathbbm{1}_n}\varphi(r)dr$, $\mathbbm{1}_n=(1,1,\cdots,1)\in\mathbb R^n$ and $\tilde{\Omega}$ is Reinhardt shadow of  $\Omega$.% is defined by (\ref{I_alpha}).
	\end{thm}
	
	\subsection{Applications of Theorem \ref{theorem1}
 -- Computation of some weighted Bergman kernels}
 
    Computation of the  Bergman  kernel function by explicit formulas is an important research direction in several complex variables.
%	The weighted Bergman kernel can almost never be computed explicitly. However, the theory would be slightly hollow if we did not compute at least one weighted Bergman kernel.
	 In this section, as applications of Theorem \ref{theorem1}, we give some examples to calculate weighted Bergman kernels.

	The first example is to compute the weighted Bergman kernel for $\mathbb{C}^n$. 
	
\begin{thm}\label{example 1}
%	{\bf Example 1.} 
For real number $\mu_1, \mu_2 > 0$, define weighted Bergman space $\mathcal{A}_{\mu_1, \mu_2}^2(\mathbb{C}^n)$ by
\begin{equation*}
	\mathcal{A}_{\mu_1, \mu_2}^2(\mathbb{C}^n)=\left\lbrace F\in \mathcal{O}(\mathbb{C}^n):	\|F\|_{\mathcal{A}_{\mu_1, \mu_2}^2(\mathbb{C}^n)}=\left(\int_{\mathbb{C}^n}|F(z)|^2e^{-\mu_1\|z\|^{\mu_2}} dV(z)\right)^{\frac1{2}}<\infty\right\rbrace .\label{anorm-2}
\end{equation*}	Then the reproducing kernel for Hilbert space $\mathcal{A}_{\mu_1, \mu_2}^2(\mathbb{C}^n)$ is
\begin{equation}\label{EX1}
	K_{\mu_1, \mu_2}(z,w)=
	\sum_{k=0}^{\infty}
	\frac{{\mu_1}^{\frac{2k+2n}{\mu_2}}\mu_2\Gamma(k+n)}{2\pi^nk!\Gamma(\frac{2k+2n}{\mu_2})}\langle z,w \rangle^{k}.
\end{equation}

%Let  positive real number $\mu_1, \mu_2 > 0$,  denote $\mathcal{A}_{\mu_1, \mu_2}^2(\mathbb{C}^n)$ the space of analytic functions $F(z)$ in $ \mathbb{C}^n$ such that
%	\begin{equation*}
%		\|F\|_{\mathcal{A}_{\mu_1, \mu_2}^2(\mathbb{C}^n)}=\left(\int_{\mathbb{C}^n}|F(z)|^2e^{-\mu_1\|z\|^{\mu_2}} dV(z)\right)^{\frac1{2}}<\infty.\label{anorm-2}
%	\end{equation*}
%	Then, the reproducing kernel for Hilbert space $\mathcal{A}_{\mu_1, \mu_2}^2(\mathbb{C}^n)$ is
%	\begin{equation}\label{EX1}
%		K_{\mu_1, \mu_2}(z,w)=
%		\sum_{k=0}^{\infty}
%		\frac{{\mu_1}^{\frac{2k+2n}{\mu_2}}\mu_2\Gamma(k+n)}{2\pi^nk!\Gamma(\frac{2k+2n}{\mu_2})}\langle z,w \rangle^{k}.		
%	\end{equation}
\end{thm}
	
%	\begin{proof}

\begin{rem}\label{remark 1}
%{\bf Remark 1.} 
When $\mu_2=2$, according to equation (\ref{EX1}), the weighted Bergman kernel of Hilbert space $\mathcal{A}_{\mu_1, 2}^2(\mathbb{C}^n)$ is
\begin{align*}
	K_{\mu_1, 2}(z,w)
	&=\sum_{k=0}^{\infty}
	\frac{{\mu_1}^{k+n}\Gamma(k+n)}{\pi^nk!\Gamma(k+n)}\langle z,w \rangle^{k}\\
	&=\left(\frac{\mu_1}{\pi}\right)^n\sum_{k=0}^{\infty}\frac{(\mu_1 \langle z,w \rangle)^k}{k!}=\left(\frac{\mu_1}{\pi}\right)^ne^{\mu_1 \langle z,w \rangle}.		
\end{align*}
\end{rem}

For given positive real number $\mu_1, \mu_2 > 0$,
the generalized Fock-Bargmann-Hartogs domain $D_{n,m}$ is a Hartogs domain defined by
\begin{align}\label{Dnm}
	D_{n,m}:= \{(z, w)\in \mathbb{C}^n \times \mathbb{C}^m : \|w\|^2 <{e}^{-\mu_1\|z\|^{\mu_2}}
\}.
\end{align}
The domain, generalizing the definition in \cite{YA}, is an unbounded, inhomogeneous strongly pseudoconvex domain in $\mathbb{C}^n \times \mathbb{C}^m$ with a smooth real-analytic boundary. We compute the
weighted Bergman kernel of $D_{n,m}$ with respect to the weight $\varphi^{\eta}$, where $\varphi(z,w) :={e}^{-\mu_1\|z\|^{\mu_2}}-\|w\|^2 $ and
$\eta>-1$.	

\begin{thm}\label{example 2}
%{\bf Example 2.}
Suppose that $D_{n,m}$  is defined by (\ref{Dnm}) and $\eta>-1$. Define weighted Bergman space $\mathcal{A}^2(D_{n,m},\varphi^{\eta})$  with  $\varphi(z,w) ={e}^{-\mu_1\|z\|^{\mu_2}}-\|w\|^2 $ by
\begin{equation*}
	\mathcal{A}^2(D_{n,m},\varphi^{\eta})=\left\lbrace F\in \mathcal{O}(D_{n,m}): 	\|F\|_{\mathcal{A}^2(D_{n,m},\varphi^{\eta})}=\left(\int_{D_{n,m}}|F|^2\varphi^{\eta} dV\right)^{\frac1{2}}<\infty\right\rbrace.
\end{equation*}
Then the reproducing kernel for Hilbert space $\mathcal{A}^2(D_{n,m},\varphi^{\eta})$ is
\begin{equation}\label{EX2}
	\sum_{k_1=0}^{\infty}\sum_{k_2=0}^{\infty}
	\frac{\mu_2\Gamma(k_2+m+\eta+1)\Gamma(k_1+n)[\mu_1(k_2+m+\eta)]^{\frac{2k_1+2n}{\mu_2}}}{2\pi^{n+m}k_1!k_2!\Gamma(\eta+1)\Gamma(\frac{2k_1+2n}{\mu_2})}\langle z,s \rangle^{k_1} \langle w,t \rangle^{k_2}.
\end{equation}

%Let $\varphi(z,w) :={e}^{-\mu_1\|z\|^{\mu_2}}-\|w\|^2 $, $(z,w)\in D_{n,m}$. For $\eta>-1$,
%denote $\mathcal{A}^2(D_{n,m},\varphi^{\eta})$
%the space of analytic functions in $ D_{n,m}$, such that
%\begin{equation*}
%	\|F\|_{\mathcal{A}^2(D_{n,m},\varphi^{\eta})}=\left(\int_{D_{n,m}}|F|^2\varphi^{\eta} dV\right)^{\frac1{2}}<\infty.
%\end{equation*}
%Then, the reproducing kernel for Hilbert space $\mathcal{A}^2(D_{n,m},\varphi^{\eta})$ is
%\begin{equation}\label{EX2}
%	\sum_{k_1=0}^{\infty}\sum_{k_2=0}^{\infty}
%	\frac{\mu_2\Gamma(k_2+m+\eta+1)\Gamma(k_1+n)[\mu_1(k_2+m+\eta)]^{\frac{2k_1+2n}{\mu_2}}}{2\pi^{n+m}k_1!k_2!\Gamma(\eta+1)\Gamma(\frac{2k_1+2n}{\mu_2})}\langle z,s \rangle^{k_1} \langle w,t \rangle^{k_2}.
%\end{equation}
\end{thm}

%\begin{proof}

\begin{rem}\label{remark 2}
When $\mu_2=0$, $D_{n,m}$ becomes
$$ \{(z, w)\in \mathbb{C}^n \times \mathbb{C}^m : \|w\|^2 <{e}^{-\mu_1\|z\|^{2}}
\},$$
which is mentioned in the introduction.  Using 
(\ref{EX2}), we obtain its kernel function (please refer to \cite{YA}):
$$
\sum_{k_1=0}^{\infty}\sum_{k_2=0}^{\infty}
\frac{\Gamma(k_2+m+\eta+1)\Gamma(k_1+n)[\mu_1(k_2+m+\eta)]^{k_1+n}}{\pi^{n+m}k_1!k_2!\Gamma(\eta+1)\Gamma(k_1+n)}\langle z,s \rangle^{k_1} \langle w,t \rangle^{k_2}.
$$
\end{rem}

\begin{thm}\label{example 3}
%{\bf Example 3.} 
%The reproducing kernel $K_{\mathcal{A}^2(V_{\eta},\varphi^{a})}$ for Hilbert space $\mathcal{A}^2(V_{\eta},\varphi^{a})$ is
%\begin{equation*}
%	K_{\mathcal{A}^2(V_{\eta},\varphi^{a})}=
%	\frac{e^{|\eta| w\bar{t}}}{\pi^{n+m+1}\Gamma(a+1)}\left(\frac{\Gamma(n+m+a+2) \sum_{j=1}^{n} \eta_{j} e^{\eta_{j} w \bar{t}} z_{j} \bar{s}_{j}}{\phi^{n+m+a+2}}+\frac{|\eta|\Gamma(n+m+a+1)}{\phi^{n+m+a+1}}\right),
%\end{equation*}
%where
%$$
%V_{\eta}=\left\{\left(z, z', w\right) \in \mathbb{C}^{n} \times \mathbb{C}^{m} \times \mathbb{C} : \sum_{j=1}^{n} e^{\eta_{j}|w|^{2}}\left|z_{j}\right|^{2}+\|z'\|^{2}<1\right\},
%$$
%\begin{equation*}
%\mathcal{A}^2(V_{\eta},\varphi^{a})=\left\lbrace F\in \mathcal{O}(V_{\eta}):	\|F\|_{\mathcal{A}^2(V_{\eta},\varphi^{a})}=\left(\int_{V_{\eta}}|F|^2\varphi^{a} dV\right)^{\frac1{2}}<\infty\right\rbrace ,%\label{anorm-2}
%\end{equation*}
%$\varphi(z, z', w) :=1-\sum_{j=1}^{n} e^{\eta_{j}|w|^{2}}\left|z_{j}\right|^{2}-\|z'\|^{2} $,  $\eta_j>0~(j=1,2,\cdots,n)$, $a>-1$ and $\phi\left(z, z', w ; s, s', t\right)=1-\sum_{j=1}^{n} e^{\eta_{j} w \bar{t}} z_{j} \bar{s}_{j}-\langle z', s' \rangle.$

 Let
$$
V_{\eta}=\left\{\left(z, z', w\right) \in \mathbb{C}^{n} \times \mathbb{C}^{m} \times \mathbb{C} : \sum_{j=1}^{n} e^{\eta_{j}|w|^{2}}\left|z_{j}\right|^{2}+\|z'\|^{2}<1\right\},
$$
$\varphi(z, z', w) :=1-\sum_{j=1}^{n} e^{\eta_{j}|w|^{2}}\left|z_{j}\right|^{2}-\|z'\|^{2} $,  $\eta_j>0~(j=1,2,\cdots,n)$, $a>-1$, and denote $\mathcal{A}^2(V_{\eta},\varphi^{a})$ the space of analytic function $F(z, z', w)$ in $V_{\eta}$ such that
\begin{equation*}
	\|F\|_{\mathcal{A}^2(V_{\eta},\varphi^{a})}=\left(\int_{V_{\eta}}|F|^2\varphi^{a} dV\right)^{\frac1{2}}<\infty.%\label{anorm-2}
\end{equation*}
 Then the reproducing kernel $K_{\mathcal{A}^2(V_{\eta},\varphi^{a})}$ for Hilbert space $\mathcal{A}^2(V_{\eta},\varphi^{a})$ is
	\begin{equation*}
%K_{\mathcal{A}^2(V_{\eta},\varphi^{a})}=
\frac{e^{|\eta| w\bar{t}}}{\pi^{n+m+1}\Gamma(a+1)}\left(\frac{\Gamma(n+m+a+2) \sum_{j=1}^{n} \eta_{j} e^{\eta_{j} w \bar{t}} z_{j} \bar{s}_{j}}{\phi^{n+m+a+2}}+\frac{|\eta|\Gamma(n+m+a+1)}{\phi^{n+m+a+1}}\right),
	\end{equation*}
where $\phi\left(z, z', w ; s, s', t\right)=1-\sum_{j=1}^{n} e^{\eta_{j} w \bar{t}} z_{j} \bar{s}_{j}-\langle z', s' \rangle$.
\end{thm}
%	\begin{proof}

\begin{rem}\label{remark 3}
%{\bf Remark 3.} 
When $a=0$, Theorem \ref{example 3} shows  the reproducing kernel of unweighted Bergman space $\mathcal{A}^2(V_{\eta})$, which is just Huo \cite[Example 4.3]{HZH}.  
%When $a=0,~\eta = 1$ and $n = m = 1$, $\mathcal{A}^2(V_{\eta})$ becomes
%$$
%(z, z', w)\in \mathbb{C}^3 : e^{|w|^2}|z|^2 +|z'|^2 < 1,
%$$
%which is mentioned in the abstract.
\end{rem}

	\section{Preliminary lemmas}\label{lemmas}
In order to prove the main results above,
%		To prove Theorem \ref{theorem1},
%		and Theorem \ref{example 1} - \ref{example 3}, 
		we need the following lemmas.

		%%%%%[Theorem 2.4.5]
		\begin{lem}\cite{HO}\label{PowerSeries}
 Let $\Omega\subset \mathbb{C}^n$ be a connected Reinhardt domain containing the origin, $D=\{z=(z_1,\dots,z_n)\in \Omega :\left|z_{j}\right|<r_{j}, j=1,\dots,n\}\subset \Omega$, and $F\in \mathcal{O}(\Omega)$. Then there exists one (and only one) power series such that
$$	
F(z)=\sum_{\alpha \in \mathbb N^n}C_\alpha(F)z^{\alpha},	
$$	
with normal convergence in $\Omega$, where
$$
C_\alpha(F)=\frac{D^{\alpha} F(0) }{  \alpha!} =\frac{1}{(2 \pi i)^{n} } \int_{|\zeta_{1}|=\rho_1}\cdots\int_{|\zeta_{n}|=\rho_n} \frac{F\left(\zeta_{1}, \cdots, \zeta_{n}\right)}{\prod_{j=1}^{n} \zeta_{j}^{\alpha_{j}+1}}  d \zeta_{1} \cdots d \zeta_{n},
$$
$0<\rho_1<r_1, \cdots, 0<\rho_n<r_n$.
		\end{lem}

Before giving the next lemma, let's introduce two Hilbert spaces.
We suppose that $\varphi(z)$ is a positive continuous function on a connected Reinhardt domain $\Omega$ containing the origin. In addition, we assume that $\varphi(z)=\varphi(|z_1|,|z_2|,\cdots,|z_n|)$ is a radial function with respect to each component $z_j \,(j=1,\cdots, n).$ Let
%For a positive continuous function $\Psi(z)$ on $\Omega$, we assume that $\Psi(z)=0$ for $z\not\in \Omega $.
%Let $\varphi (r) $ be a positive continuous function on$\tilde{\Omega}$. Let
				\begin{equation}\label{I_alpha}
					I(\alpha)=(2\pi)^n\int_{\tilde{\Omega}}r^{2\alpha+\mathbbm{1}_n}\varphi(r)dr,
							\end{equation}
			where $\mathbbm{1}_n=(1,1,\cdots,1)\in\mathbb R^n$, $\tilde{\Omega}$ is the Reinhardt shadow of $\Omega$.
			
The weighted $l^p_I $ space is the set of sequences $C:=\{C_\alpha\}_{\alpha\in \mathbb N^n}$ such that
			$$
			\|C\|_{l^p_I}=\sum_{\alpha\in \mathbb N^n}|C_\alpha|^pI(\alpha)<\infty.
			$$
		We also denote
		$\mathcal{A}^2_\varphi$ the
		subspace of 
		$L^2_{\varphi}(\Omega)$,
		which consists of all holomorphic functions $f$ in the domain $\Omega$ and $f\in L^2_{\varphi}(\Omega)$. We see that $\mathcal{A}^2_\varphi$ and $l^2_I$ are Hilbert spaces with the inner product
			$$
			\langle F,G\rangle_{\varphi}=\int_{\Omega}F(z)\overline{G(z)}\varphi(|z_1|,|z_2|,\cdots,|z_n|)dV(z),
			$$
			for $F,G\in \mathcal{A}^2_\varphi(\Omega)$  and
			$$
			\langle C,\tilde{C}\rangle_I=\sum_{\alpha\in \mathbb N^n}C_\alpha \overline{\tilde{C}_\alpha} I(\alpha),
			$$
			for $C:=\{C_\alpha\}_{\alpha\in \mathbb N^n} ,\tilde{C}:=\{\tilde{C}_\alpha\}_{\alpha\in \mathbb N^n}\in l^2_I$, respectively.
%%%%%%
		\begin{lem}\label{lem1}
			The transform $T:F \longmapsto \{C_\alpha(F)\}_{\alpha\in \mathbb N^n}$ is an isometry from $\mathcal{A}^2_\varphi$ to $l^2_I$ preserving the Hilbert space norms, i.e.,
	$$
	\|F\|_{\mathcal{A}_{\varphi}^2}=\|T(F)\|_{l_I^2},
	$$
			where
			$
	        C_\alpha(F)=\frac{D^{\alpha} F(0) }{  \alpha!}.
		    $   
			
		\end{lem}
		\begin{proof}
			Firstly, it's easy to show that $T:\mathcal{A}^2_\varphi\rightarrow l^2_I$ is a linear transformation. 
	%		???, ?$\varphi_\alpha(z)=z^{\alpha},$???? $\{\frac{\varphi_\alpha(z)}{\sqrt{I(\alpha)}}:\alpha\in \mathbb N^n\}$
			Further, let $\varphi_\alpha(z)=z^{\alpha},$ we prove that the form of a complete orthonormal basis on Reinhardt domain $\Omega$ is
			$\{\frac{\varphi_\alpha(z)}{\sqrt{I(\alpha)}}:\alpha\in \mathbb N^n\}$, where $I(\alpha)$ is defined by (\ref{I_alpha}). %$\{\frac{z^{\alpha}}{\sqrt{I(\alpha)}}:\alpha\in \mathbb N^n\}$.
			As in the case of $\alpha\neq\beta$,
				\begin{eqnarray*}
				\int_\Omega z^{\alpha}	\overline{z}^{\beta}\varphi(|z_1|,|z_2|,\cdots,|z_n|)dV(z)				
				=\int_{0}^{2\pi}\cdots\int_{0}^{2\pi}e^{i\langle \alpha-\beta,\theta\rangle}d\theta\int_{\tilde{\Omega}}r^{\alpha+\beta+\mathbbm{1}_n}\varphi(r)dr=0,
			\end{eqnarray*}
	where $
	\tilde{\Omega}$ is the Reinhardt shadow of $\Omega$.
		If $\alpha=\beta$, we have
		\begin{eqnarray*}
		\|\varphi_\alpha\|^2_{\mathcal{A}^2_\varphi}=\int_\Omega |z^{\alpha}|^2	\varphi(|z_1|,|z_2|,\cdots,|z_n|)dV(z)	=(2\pi)^n\int_{\tilde{\Omega}}r^{2\alpha+\mathbbm{1}_n}\varphi(r)dr=I(\alpha),
	\end{eqnarray*}
	where $\mathbbm{1}_n=(1,1,\cdots,1)\in\mathbb R^n$.	From Lemma \ref{PowerSeries}, for any $F\in \mathcal{A}^2_\varphi(\Omega)$,  there exists one (and only one) power series such that
	$$	
	F(z)=\sum_{\alpha \in \mathbb N^n}C_\alpha(F)z^{\alpha},	
	$$
		where
	$
	C_\alpha(F)=\frac{D^{\alpha} F(0) }{  \alpha!}.
	$  
	Orthogonality of $\{z^{\alpha}:\alpha\in \mathbb N^n\}$ of $\mathcal{A}^2_\varphi$ gives
		\begin{eqnarray*}
		\|F\|^2_{\mathcal{A}^2_\varphi}	&=&\|F-\sum_{|\alpha|  \leqslant N}C_\alpha(F)\varphi_\alpha\|^2_{\mathcal{A}^2_\varphi}+
		\|\sum_{|\alpha|  \leqslant N}C_\alpha(F)\varphi_\alpha\|^2_{\mathcal{A}^2_\varphi}\\
		&=&\|F-\sum_{|\alpha|  \leqslant N}C_\alpha(F)\varphi_\alpha\|^2_{\mathcal{A}^2_\varphi}+
		\sum_{|\alpha| \leqslant N}|C_\alpha(F)|^2I(\alpha).
	\end{eqnarray*}	
	Then
		$
		\sum_{|\alpha| \leqslant N}|C_\alpha(F)|^2I(\alpha)\leqslant\|F\|^2_{\mathcal{A}^2_\varphi}	,
	    $
	    which follows that	
		$$
	\|C_\alpha(F)\|_{l^2_I}=\sum_{\alpha\in \mathbb N^n}|C_\alpha(F)|^2I(\alpha)
	\leqslant\|F\|^2_{\mathcal{A}^2_\varphi}.
	$$
	On the other hand, according to the Fatou's lemma,
		\begin{eqnarray*}
		\|F-\sum_{|\alpha|  \leqslant N}C_\alpha(F)\varphi_\alpha\|^2_{\mathcal{A}^2_\varphi}
		&=&\|\lim_{M\rightarrow \infty}\sum_{N<|\alpha|<M  }C_\alpha(F)\varphi_\alpha\|^2_{\mathcal{A}^2_\varphi}\\
		&\leqslant&\varliminf_{M\rightarrow \infty}\|\sum_{N<|\alpha|<M  }C_\alpha(F)\varphi_\alpha\|^2_{\mathcal{A}^2_\varphi}\\
		&=&\varliminf_{M\rightarrow \infty}\sum_{N<|\alpha|<M}|C_\alpha(F)|^2I(\alpha).
	\end{eqnarray*}
	An immediately consequence is
		\begin{eqnarray*}
		\lim_{N\rightarrow \infty}\|F-\sum_{|\alpha|  \leqslant N}C_\alpha(F)\varphi_\alpha\|^2_{\mathcal{A}^2_\varphi}=0.
	\end{eqnarray*}	
	Therefore, $\{\frac{\varphi_\alpha(z)}{\sqrt{I(\alpha)}}:\alpha\in \mathbb N^n\}$ is complete and $ \|F\|_{\mathcal{A}_{\varphi}^2}=\|T(F)\|_{l_I^2}.$
		\end{proof}

%	Let us begin by introducing the following lemmas.
		
	\begin{lem}\label{lem2}
	For $\alpha=(\alpha_1,\alpha_2,\cdots,\alpha_n), \alpha_j>-1~(j=1,\cdots,n)$, the following multiple integral exists:
	\begin{eqnarray}\label{equation lemma2 sphere integral}
	\int_{\mathbb S_n}|x_1|^{2\alpha_1+1}\cdots|x_n|^{2\alpha_n+1} dx=\frac{2\alpha!}{\Gamma(|\alpha|+n)},
	\end{eqnarray}
     where $\mathbb S_n$ is the unit sphere in $\mathbb R^n$, $|\alpha| = \alpha_1 +\cdots+ \alpha_n$ and $\alpha! =\Gamma(\alpha_1+1)\cdots \Gamma(\alpha_n+1)$.
    \begin{proof}
    We evaluate the integral
    $$
    I=\int_{\mathbb R^n}|y_1|^{2\alpha_1+1}\cdots|y_n|^{2\alpha_n+1} e^{-|y|^{2}} dy
    $$
    by two different methods. First,
    \begin{align*}
    	I &=\prod_{k=1}^{n} \int_{\mathbb{R}} |y_k|^{2\alpha_{k}+1} e^{-y_k^{2}} d y_k
    	=2\prod_{k=1}^{n} \int_{0}^\infty y_k^{2\alpha_{k}+1} e^{-y_k^{2}} d y_k \\
    	&=\prod_{k=1}^{n} \int_{0}^{\infty} y_k^{\alpha_k} e^{-y_k} d y_k =\alpha !.
    \end{align*}
    Then, integration in polar coordinates gives
     \begin{align*}
    	I
    	&=\int_{0}^{\infty}r^{2|\alpha|+2n-1}e^{-r^{2}}dr\int_{\mathbb S_n}|x_1|^{2\alpha_1+1}\cdots|x_n|^{2\alpha_n+1} dx\\
    	&=\frac12\int_{0}^{\infty}t^{|\alpha|+n-1}e^{-t}dt\int_{\mathbb S_n}|x_1|^{2\alpha_1+1}\cdots|x_n|^{2\alpha_n+1} dx \\
    	&=\frac{\Gamma(|\alpha|+n)}{2}\int_{\mathbb S_n}|x_1|^{2\alpha_1+1}\cdots|x_n|^{2\alpha_n+1} dx.
    \end{align*}
Comparing the two answers, we obtain
	\begin{eqnarray*}
	\int_{\mathbb S_n}|x_1|^{2\alpha_1+1}\cdots|x_n|^{2\alpha_n+1} dx=\frac{2\alpha!}{\Gamma(|\alpha|+n)}.
\end{eqnarray*}		
    \end{proof}
	\end{lem}
	
	\begin{lem}\cite{ZKH}\label{lem3}
	For a multi-index $m = (m_1, \cdots, m_n)$ of nonnegative integers $m_i$, a positive integer $N$, we have the multi-nomial formula
	\begin{eqnarray}\label{equation lemma3 multiindex}	
		(z_1 + \cdots + z_n)^N =\sum_{|m|=N}\frac{N!}{m!}z^m.
	\end{eqnarray}
 where  $
 z^m=z_1^{m_1} \ldots z_n^{m_n},
 $ $|m| = m_1 +\cdots+ m_n$ and $m! =\Gamma(m_1+1)\cdots \Gamma(m_n+1)$.
   \end{lem}

\section{Proof of the theorems}\label{proof}
%	Let us now prove Theorem \ref{theorem1}.
Now we are ready to prove Theorems \ref{theorem1}, \ref{example 1}, \ref{example 2} and \ref{example 3}.

\begin{proof}[\textbf{Proof of Theorem $\mathbf{\ref{theorem1}}$}]
		From Lemma \ref{PowerSeries}, for the kernel $K_{w}(z)=K(z,w)\in \mathcal{A}^2_\varphi$, $D=\left\{z: \left|z_{j}\right|<r_{j}\right\}\subset \Omega$, there exists one (and only one) power series such that
		$$	
		K_{w}(z)=\sum_{\alpha \in \mathbb N^n}C_\alpha(K_{w}(z))z^{\alpha}
		$$	
		with normal convergence in $\Omega$, where
		$$
		{C_\alpha(K_{w})}
		=\frac{1}{(2 \pi i)^{n} } \int_{|\zeta_{1}|=\rho_1}\cdots\int_{|\zeta_{n}|=\rho_n} \frac{K_{w}\left(\zeta_{1}, \cdots, \zeta_{n}\right)}{\prod_{j=1}^{n} \zeta_{j}^{\alpha_{j}+1}}  d \zeta_{1} \cdots d \zeta_{n},
		$$
		$0<\rho_1<r_1, \cdots, 0<\rho_n<r_n$.
		
		Lemma \ref{lem1} claims that the transform $T:F \longmapsto \{C_\alpha(F)\}_{\alpha\in \mathbb N^n}$ is an isometry from $\mathcal{A}^2_\varphi$ to $l^2_I$ preserving the Hilbert space norms. Using the polarization identity of $\|F\|_{\mathcal{A}^2_\varphi}=\|T(F)\|_{l_I^2}$, it then follows that the inner product is also preserved. Hence we have
		\begin{equation*}
	F(w)=\langle F,K_{w}\rangle_{\varphi}=\langle T(F),T(K_{w})\rangle_{I}
	=\sum_{\alpha\in \mathbb N^n}C_\alpha(F) \overline{C_\alpha(K_{w})} I(\alpha).
\end{equation*}
		On the other hand,
		$$	
		F(w)=\sum_{\alpha \in \mathbb N^n}C_\alpha(F)w^{\alpha}.	
		$$
 Therefore 
		$$
		\sum_{\alpha\in \mathbb N^n}C_\alpha(F)\left( \overline{C_\alpha(K_{w})} I(\alpha)-w^{\alpha}\right)=0
		$$
		holds for every $F\in \mathcal{A}^2_\varphi$. In particular, for any $\alpha \in \mathbb N^n$, let $F_{\alpha}(w)=w^{\alpha}$, then $$C_\beta(F_{\alpha})=\left\{
		\begin{array}{rcl}
			1,       &      & {\alpha=\beta,}\\
			0,     &      & {\alpha\neq\beta,}
		\end{array} \right.$$
		 which implies that $\overline{C_\alpha(K_{w})} I(\alpha)=w^{\alpha}$. Then, 	
		 \begin{equation}
		 	C_\alpha(K_{w})=I^{-1}(\alpha)\overline{w}^{\alpha}.\label{C}
		 \end{equation}
		Hence,
		\begin{equation*}
			K_{w}(z)=K(z,w)
			=\sum_{\alpha \in \mathbb N^n}C_\alpha(K_{w})z^{\alpha}
			=\sum_{\alpha\in \mathbb N^n} I^{-1}(\alpha)z^{\alpha}\overline{w}^{\alpha}.
		\end{equation*}
	Note that the Bergman kernel is uniquely characterized by the following three properties \cite[Proposition 1.1.6.]{SGK1}:\\
	(i) $K(z,w)=\overline{K(w,z)}$  for all $z,w\in \Omega$;\\
	(ii) $K(z,w)$ reproduces every element in $\mathcal{A}^2_\varphi(\Omega)$ in the following sense \begin{equation*}
		F(z)=\int_\Omega K(z,w)F(w)\varphi(|w_1|,|w_2|,\cdots,|w_n|)dV(w),
	\end{equation*}
	for every $F\in \mathcal{A}^2_\varphi$;\\
	(iii)$K_{w}\in \mathcal{A}^2_\varphi$ for all $w\in \Omega$, where $K_{w}(z)=K(z,w)$.
	
	Now we prove that (\ref{kernelform}) admits these properties.
	We first prove the equation in (i).
	\begin{eqnarray*}
		\overline{K(w,z)}&=&\overline{\sum_{\alpha\in \mathbb N^n} I^{-1}(\alpha)w^{\alpha}\overline{z}^{\alpha}}\\
		&=&\sum_{\alpha\in \mathbb N^n} I^{-1}(\alpha)z^{\alpha}\overline{w}^{\alpha}=K(z,w),
	\end{eqnarray*}
	which means that (i) holds for the Bergman kernel in the form of (\ref{kernelform}). We then show that $K(z,w)$ reproduces every element in $\mathcal{A}^2_\varphi$. For $F(z),K_{w}(z)\in \mathcal{A}^2_\varphi$,
	the polarization identity and (\ref{C}) implies that
	\begin{eqnarray*}
	\langle F,K_{w}\rangle_{\mathcal{A}^2_{\varphi}}
	&=&\left \langle T(F),T(K_{w})\right \rangle_{l^2_I}=\sum_{\alpha \in \mathbb N^n}C_\alpha(F)I^{-1}(\alpha)w^{\alpha}I(\alpha)\\
	&=&\sum_{\alpha \in \mathbb N^n}C_\alpha(F)w^{\alpha}
	=F(w).
\end{eqnarray*}
	Hence, the second property is proved.
	
	Finally, for fixed $w_0=(r_1e^{i\theta_1},\cdots,r_ne^{i\theta_n})\in \Omega$, we prove that $K_{w_0}(z)\in \mathcal{A}^2_\varphi(\Omega) $.    $\tilde{\Omega}$ is the Reinhardt shadow of $\Omega$. There exists $\delta>0$ such that
    $r_0:=(r_1, \cdots, r_n)\in P_{\delta}\subset \tilde{\Omega}$, where $P_{\delta}=[0,\delta]\times \cdots \times [0,\delta]\subset\mathbb R^n$.
%	 $r_0 \in P_{\delta}\subset \tilde{\Omega}$, where $P_{\delta}=[0,\delta]^n\subset\mathbb R^n$. 
Let $\varepsilon=\min\{\varphi(r):r\in P_{\delta}\}>0$, then,
	\begin{equation*}
		I(\alpha)=(2\pi)^n\int_{\tilde{\Omega}}r^{2\alpha+\mathbbm{1}_n}\varphi(r)dr
		\geq (2\pi)^n\varepsilon\int_{P_{\delta}}r^{2\alpha+\mathbbm{1}_n}dr
		=(2\pi)^n\varepsilon\prod_{k=1}^n\frac{\delta^{2\alpha_k+2}}{2\alpha_k+2}.
	\end{equation*}
	Therefore, again by the polarization identity and (\ref{C}),
	\begin{eqnarray*}
		\langle K_{w_0},K_{w_0}\rangle_{\mathcal{A}^2_{\varphi}}
		&=&\left \langle T(K_{w}),T(K_{w})\right \rangle_{l^2_I}
		=\sum_{\alpha \in \mathbb N^n}\left|\frac{\overline{w_0}^{\alpha}}{I(\alpha)}\right |^2I(\alpha)
		=\sum_{\alpha \in \mathbb N^n}I^{-1}(\alpha)|w_0^{\alpha}|^2\\
		&\leq&\frac{1}{(2\pi)^n\varepsilon}\sum_{\alpha \in \mathbb N^n}\frac{|r_1|^{2\alpha_{1}}\cdots|r_n|^{2\alpha_{n}}}{\prod_{k=1}^n\frac{\delta^{2\alpha_k+2}}{2\alpha_k+2}} <\infty.
	\end{eqnarray*}
	Therefore, $K_{w_0}(z)\in \mathcal{A}^2_\varphi$ for $z \in \Omega$.
	\end{proof}
	
	\begin{proof}[\textbf{Proof of Theorem $\mathbf{\ref{example 1}}$}]
	 We first compute I($\alpha$) as follows:
	 	\begin{eqnarray*}
	 I(\alpha)=(2\pi)^n\int_{\tilde{\Omega}}r^{2\alpha+\mathbbm{1}_n}\varphi(r)dr,
	  	\end{eqnarray*}
	 where $\tilde{\Omega}= \mathbb{R}_{+}^n:=\{r=(r_1,\cdots,r_n):r_j\geq0,j=1,2,\cdots,n\}$ is the Reinhardt shadow of $\mathbb{C}^n$, $\varphi(r)=e^{-\mu_1\|r\|^{\mu_2}}$. Then, from Lemma \ref{lem2}, we obtain
	 	\begin{align}\label{Ialpha}
		I(\alpha)
		&=(2\pi)^n\int_{\tilde{\Omega}}r^{2\alpha+\mathbbm{1}_n}e^{-\mu_1\|r\|^{\mu_2}} dr\nonumber\\
		&=(2\pi)^n\int_0^{\infty}\rho^{2|\alpha|+2n-1}e^{-\mu_1\rho^{\mu_2}}d\rho \cdot  \int_{\mathbb S^{+}_n}\xi^{2\alpha+\mathbbm{1}_n} d\xi\nonumber\\
		&=(2\pi)^n  \int_0^{\infty}\rho^{2|\alpha|+2n-1}e^{-\mu_1\rho^{\mu_2}}d\rho \cdot \frac{1}{2^n}\int_{\mathbb S_n}|\xi_1|^{2\alpha_1+1}\cdots|\xi_n|^{2\alpha_n+1}  d\xi\nonumber\\
		&=\frac{2\pi^n\alpha!}{\Gamma(|\alpha|+n)}\int_0^{\infty}\rho^{2|\alpha|+2n-1}e^{-\mu_1\rho^{\mu_2}}d\rho,
	    \end{align}
	where $\mathbb S^{+}_n=\{\xi \in\mathbb R^n_{+}:\|\xi\|=1\}$. By substituting $t = \mu_1\rho^{\mu_2}$
	 to the last line of (\ref{Ialpha}), we obtain 
		\begin{align*}
		I(\alpha)
		&=\frac{2\pi^n\alpha!}{\Gamma(|\alpha|+n)}\cdot\frac{1}{{\mu_1}^{\frac{2|\alpha|+2n}{\mu_2}}\mu_2}\int_0^{\infty}t^{\frac{2|\alpha|+2n}{\mu_2}-1}e^{-t}d\rho\\
		&=\frac{2\pi^n\alpha!\Gamma(\frac{2|\alpha|+2n}{\mu_2})}{\Gamma(|\alpha|+n){\mu_1}^{\frac{2|\alpha|+2n}{\mu_2}}\mu_2}.
	    \end{align*}
		Then based on the form of the kernel of Theorem \ref{theorem1}, we obtain
	 	\begin{align*}
		K_{\mu_1, \mu_2}(z,w)
		&=\sum_{\alpha\in \mathbb N^n} I^{-1}(\alpha)z^{\alpha}\overline{w}^{\alpha}
		=\sum_{\alpha\in \mathbb N^n}
		\frac{\Gamma(|\alpha|+n){\mu_1}^{\frac{2|\alpha|+2n}{\mu_2}}\mu_2}{2\pi^n\alpha!\Gamma(\frac{2|\alpha|+2n}{\mu_2})}z^{\alpha}\overline{w}^{\alpha}\\
		&=\frac{\mu_2}{2\pi^n}\sum_{k=0}^\infty\sum_{|\alpha|=k}
		\frac{\Gamma(|\alpha|+n){\mu_1}^{\frac{2|\alpha|+2n}{\mu_2}}}{\alpha!\Gamma(\frac{2|\alpha|+2n}{\mu_2})}z^{\alpha}\overline{w}^{\alpha}\\
		&=\frac{\mu_2}{2\pi^n}\sum_{k=0}^\infty\frac{\Gamma(k+n){\mu_1}^{\frac{2k+2n}{\mu_2}}}{\Gamma(\frac{2k+2n}{\mu_2})k!}\sum_{|\alpha|=k}\frac{k!}{\alpha!}z^{\alpha}\overline{w}^{\alpha}.
		  \end{align*}
	According to (\ref{equation lemma3 multiindex}), we can know that $\sum_{|\alpha|=k}\frac{k!}{\alpha!}z^{\alpha}\overline{w}^{\alpha}=\langle z,w \rangle^k$, then 
	  \begin{align*}
		K_{\mu_1, \mu_2}(z,w)
		&=\frac{\mu_2}{2\pi^n}\sum_{k=0}^\infty\frac{\Gamma(k+n){\mu_1}^{\frac{2k+2n}{\mu_2}}}{\Gamma(\frac{2k+2n}{\mu_2})k!}\langle z,w \rangle^k.
	    \end{align*}
		\end{proof}

\begin{proof}[\textbf{Proof of Theorem $\mathbf{\ref{example 2}}$}]
	From the definition of the region, we can see that
	$\tilde{\Omega}=\{(r, \rho)\in \mathbb R_{+}^n \times \mathbb{R}_{+}^m : \|\rho\|^2 <e^{-\mu_1\|r\|^{\mu_2}}
	\}$ is the Reinhardt shadow of $D_{n,m}$, $\varphi(r, \rho)=(e^{-\mu_1\|r\|^{\mu_2}}-\|\rho\|^2)^\eta$. Then
	\begin{align}\label{Iab}
		&~~I(\alpha,\beta)\\
		=&(2\pi)^{n+m}\int_{\tilde{\Omega}}r^{2\alpha+\mathbbm{1}_n}\rho^{2\beta+\mathbbm{1}_m}\varphi(r,\rho)drd\rho  \nonumber\\
		=&\frac{(2\pi)^{n+m}}{2^{n+m}}\int_{\mathbb{R}^n}\int_{\|\rho\|^2 <e^{-\mu_1\|r\|^{\mu_2}}}|r_1|^{2\alpha_1+1}\cdots|r_n|^{2\alpha_n+1}|\rho_1|^{2\beta_1+1}\cdots|\rho_m|^{2\beta_m+1}\varphi(r,\rho)d\rho dr   \nonumber\\
		=&\pi^{n+m}\int_{\mathbb{R}^n}|r_1|^{2\alpha_1+1}\cdots|r_n|^{2\alpha_n+1}\int_{\|\rho\|^2 <e^{-\mu_1\|r\|^{\mu_2}}}|\rho_1|^{2\beta_1+1}\cdots|\rho_m|^{2\beta_m+1}\varphi(r,\rho)d\rho dr.\nonumber
	\end{align}	
	Let us first calculate the integral inside:
	\begin{align*}
		I_1(r)
		&:=\int_{\|\rho\|^2 <e^{-\mu_1\|r\|^{\mu_2}}}|\rho_1|^{2\beta_1+1}\cdots|\rho_m|^{2\beta_m+1}({e}^{-\mu_1\|r\|^{\mu_2}}-\|\rho\|^2)^\eta d\rho\\
		&=\int_{0}^{\sqrt{e^{-\mu_1\|r\|^{\mu_2}} }}\tilde{\rho}^{2|\beta|+2m-1}({e}^{-\mu_1\|r\|^{\mu_2}}-\tilde{\rho}^2)^\eta d\tilde{\rho} \int_{\mathbb S_m}|\xi_1|^{2\alpha_1+1}\cdots|\xi_m|^{2\alpha_m+1} d\xi.
	\end{align*}
	Let $\tilde{\rho}=\sqrt{{e}^{-\mu_1\|r\|^{\mu_2}}}t,$ and note that
	\begin{align*}
		&~~\int_{0}^{\sqrt{{e}^{-\mu_1\|r\|^{\mu_2}} }}\tilde{\rho}^{2|\beta|+2m-1}({e}^{-\mu_1\|r\|^{\mu_2}}-\tilde{\rho}^2)^\eta d\tilde{\rho}\\
		=&e^{-\mu_1(|\beta|+m+\eta)\|r\|^{\mu_2}}\int_{0}^{1}t^{2|\beta|+2m-1}(1-t^2)^\eta dt\\
		=&e^{-\mu_1(|\beta|+m+\eta)\|r\|^{\mu_2}}\frac{B(|\beta|+m,\eta+1)}{2}.
	\end{align*}
	Now, using equation (\ref{equation lemma2 sphere integral}), we see that
	\begin{align*}
		I_1(r)
		&=e^{-\mu_1(|\beta|+m+\eta)\|r\|^{\mu_2}}\frac{B(|\beta|+m,\eta+1)}{2}\frac{2\beta!}{\Gamma(|\beta|+m)}\\
		&=\frac{\Gamma(\eta+1)\beta!}{\Gamma(|\beta|+m+\eta+1)}e^{-\mu_1(|\beta|+m+\eta)\|r\|^{\mu_2}}.
	\end{align*}
	With this formula, we see immediately that
	\begin{align*}
		&~~I(\alpha,\beta)\\
		=&\frac{\Gamma(\eta+1)\beta!\pi^{n+m}}{\Gamma(|\beta|+m+\eta+1)}\int_{\mathbb{R}^n}|r_1|^{2\alpha_1+1}\cdots|r_n|^{2\alpha_n+1}e^{-\mu_1(|\beta|+m+\eta)\|r\|^{\mu_2}}dr\\
		=&\frac{\Gamma(\eta+1)\beta!\pi^{n+m}}{\Gamma(|\beta|+m+\eta+1)}\int_{0}^{\infty}\tilde{r}^{2|\alpha|+2n-1}e^{-\mu_1(|\beta|+m+\eta)\tilde{r}^{\mu_2}} d\tilde{r} \int_{\mathbb S_n}|\xi_1|^{2\alpha_1+1}\cdots|\xi_n|^{2\alpha_n+1}  d\xi.
	\end{align*}
	Let $t=\mu_1(|\beta|+m+\eta)\tilde{r}^{\mu_2}$, then
	\begin{align*}
		\int_{0}^{\infty}\tilde{r}^{2|\alpha|+2n-1}e^{-\mu_1(|\beta|+m+\eta)\tilde{r}^{\mu_2}} d\tilde{r}
		&=\frac{1}{\mu_2[\mu_1(|\beta|+m+\eta)]^{\frac{2|\alpha|+2n}{\mu_2}}}\int_{0}^{\infty}t^{\frac{2|\alpha|+2n}{\mu_2}-1}e^{-t}dt\\
		&=\frac{\Gamma(\frac{2|\alpha|+2n}{\mu_2})}{\mu_2[\mu_1(|\beta|+m+\eta)]^{\frac{2|\alpha|+2n}{\mu_2}}}.
	\end{align*}
	Once again, using equation (\ref{equation lemma2 sphere integral}), we get
	\begin{align*}	
		I(\alpha,\beta)=\frac{2\alpha!\beta!\Gamma(\eta+1)\Gamma(\frac{2|\alpha|+2n}{\mu_2})\pi^{n+m}}{\mu_2\Gamma(|\beta|+m+\eta+1)\Gamma(|\alpha|+n)[\mu_1(|\beta|+m+\eta)]^{\frac{2|\alpha|+2n}{\mu_2}}}.
	\end{align*}
	Finally, according to Theorem \ref{theorem1},
	\begin{align*}
		&~~K_{\mathcal{A}^2(D_{n,m},\varphi^{\eta})}((z,w),(s,t))\\
		=&\sum_{\alpha\in \mathbb N^n, \beta\in \mathbb N^m} I^{-1}(\alpha,\beta)z^{\alpha}w^{\beta}\overline{{s}^{\alpha}t^{\beta}}\\
		=&\sum_{\alpha\in \mathbb N^n, \beta\in \mathbb N^m}
		\frac{\mu_2\Gamma(|\beta|+m+\eta+1)\Gamma(|\alpha|+n)[\mu_1(|\beta|+m+\eta)]^{\frac{2|\alpha|+2n}{\mu_2}}}{2\pi^{n+m}\alpha!\beta!\Gamma(\eta+1)\Gamma(\frac{2|\alpha|+2n}{\mu_2})}z^{\alpha}w^{\beta}\overline{{s}^{\alpha}t^{\beta}}\\
		=&\sum_{\beta\in \mathbb N^m}
		\frac{\mu_2\Gamma(|\beta|+m+\eta+1)[\mu_1(|\beta|+m+\eta)]^{\frac{2|\alpha|+2n}{\mu_2}}w^{\beta}\overline{t}^{\beta}}{2\pi^{n+m}\beta!\Gamma(\eta+1)}
		\sum_{k_1=0}^{\infty}\frac{\Gamma(k_1+n)}{\Gamma(\frac{2k_1+2n}{\mu_2})k_1!}\sum_{|\alpha|=k_1}\frac{k_1!}{\alpha!}z^{\alpha}\overline{s}^{\alpha}\\
		=&\sum_{k_1=0}^{\infty}\frac{\mu_2\Gamma(k_1+n)[\mu_1(|\beta|+m+\eta)]^{\frac{2|\alpha|+2n}{\mu_2}}}{2\pi^{n+m}\Gamma(\frac{2k_1+2n}{\mu_2})k_1!\Gamma(\eta+1)}\langle z,s \rangle^{k_1}
		\sum_{\beta\in \mathbb N^m}
		\frac{\Gamma(|\beta|+m+\eta+1)w^{\beta}\overline{t}^{\beta}}{\beta!}\\
		=&\sum_{k_1=0}^{\infty}\sum_{k_2=0}^{\infty}
		\frac{\mu_2\Gamma(k_2+m+\eta+1)\Gamma(k_1+n)[\mu_1(k_2+m+\eta)]^{\frac{2k_1+2n}{\mu_2}}}{2\pi^{n+m}k_1!k_2!\Gamma(\eta+1)\Gamma(\frac{2k_1+2n}{\mu_2})}\langle z,s \rangle^{k_1} \langle w,t \rangle^{k_2}.
	\end{align*}
\end{proof}

\begin{proof}[\textbf{Proof of Theorem $\mathbf{\ref{example 3}}$}]
Let $\tilde{\Omega}=\left\{(r, r', \rho)\in \mathbb R_{+}^n \times \mathbb R_{+}^m \times\mathbb{R}_{+} : \sum_{j=1}^{n} e^{\eta_{j}\rho^{2}}r_{j}^{2}+\|r'\|^{2} <1
\right\}$ be the Reinhardt shadow of $V_{\eta}$. Put $\varphi(r, r', \rho)=(1-\sum_{j=1}^{n} e^{\eta_{j}\rho^{2}}r_{j}^{2}-\|r'\|^{2})^a$. Then,
\begin{align*}
	&~~I(\alpha,\beta,\gamma)\\
=&(2\pi)^{n+m+1}\int_{\tilde{\Omega}}r^{2\alpha+\mathbbm{1}_n}{r'}^{2\beta+\mathbbm{1}_m}\rho^{2\gamma+1}\varphi(r,r',\rho)drdr'd\rho  \nonumber\\
=&(2\pi)^{n+m+1}\int_{0}^{\infty}
\int_{\mathcal{H}}
	r^{2\alpha+\mathbbm{1}_n}{r'}^{2\beta+\mathbbm{1}_m}\varphi(r,r',\rho)drdr'\rho^{2\gamma+1}d\rho,
\end{align*}	
where $\mathcal{H}={\{(r,r')\in \mathbb R_{+}^n \times \mathbb R_{+}^m : \sum_{j=1}^{n} e^{\eta_{j}\rho^{2}}r_{j}^{2}+\|r'\|^{2} <1
	\}}$.
Let us first calculate the integral inside as follows:
\begin{align}\label{I1}
	I_1(\rho)
	&:=\int_{\mathcal{H}}
%	{\{(r,r')\in \mathbb R_{+}^n \times \mathbb R_{+}^m : \sum_{j=1}^{n} e^{\eta_{j}\rho^{2}}r_{j}^{2}+\|r'\|^{2} <1		\}}
	r_{1}^{2\alpha_1+1}\cdots r_{n}^{2\alpha_n+1}{r'}^{2\beta+\mathbbm{1}_m}(1-\sum_{j=1}^{n} e^{\eta_{j}\rho^{2}}r_{j}^{2}-\|r'\|^{2})^adrdr'.
\end{align}	
Performing variable substitution on (\ref{I1}) with $t_j=e^{\frac{\eta_{j}\rho^{2}}{2}}r_{j}$, $j=1,\cdots,n$. Then%We then obtain
\begin{align*}
	&~~I_1(\rho)\\
	=&\int_{\mathbb{B}_{n+m}^{+}}
	e^{-\frac{\rho^2}{2}(2\langle\alpha , \eta\rangle +|\eta|)}t^{2\alpha+\mathbbm{1}_n}{r'}^{2\beta+\mathbbm{1}_m}(1-\|t\|^{2}-\|r'\|^{2})^ae^{-\frac{\rho^2}{2}|\eta|}dtdr'\\
	=&e^{-\rho^2(\langle\alpha , \eta\rangle +|\eta|)}\int_{\mathbb{B}_{n+m}^{+}}
	t^{2\alpha+\mathbbm{1}_n}{r'}^{2\beta+\mathbbm{1}_m}(1-\|t\|^{2}-\|r'\|^{2})^adtdr'
	=C_{\alpha,\beta,a} \cdot e^{-\rho^2(\langle\alpha , \eta\rangle +|\eta|)},
\end{align*}
where $\mathbb B^{+}_{n+m}=\{(x, y)\in\mathbb R_{+}^n \times \mathbb R_{+}^m:\|x\|^2+\|y \|^2<1\}$ and
\begin{align*}
    &~~C_{\alpha,\beta,a}\\
    =&\int_{\mathbb{B}_{n+m}^{+}}
    t^{2\alpha+\mathbbm{1}_n}{r'}^{2\beta+\mathbbm{1}_m}(1-\|t\|^{2}-\|r'\|^{2})^adtdr'\\
	=&\int_0^{1}x^{2(|\alpha|+|\beta|+n+m)-1}(1-x^2)^a dx \int_{\mathbb S^{+}_{n+m}}\xi^{2\alpha+\mathbbm{1}_n}
	\zeta^{2\beta+\mathbbm{1}_m} d\xi d\zeta\\
	=&\frac{B(|\alpha|+|\beta|+n+m, a+1)}{2^{n+m+1}}%\cdot \frac{1}{2^{n+m}}
	 \int_{\mathbb S_{n+m}}|\xi_1|^{2\alpha_1+1}\cdots|\xi_n|^{2\alpha_n+1}|\zeta_1|^{2\beta_1+1}\cdots|\zeta_n|^{2\beta_m+1}  d\xi d\zeta\\
	=&\frac{\Gamma(|\alpha|+|\beta|+n+m)\Gamma(a+1)}{2^{n+m+1}\Gamma(|\alpha|+|\beta|+n+m+a+1)}\cdot
	\frac{2\alpha!\beta!}{\Gamma(|\alpha|+|\beta|+n+m)}\\
	=&\frac{\Gamma(a+1)\alpha!\beta!}{2^{n+m}\Gamma(|\alpha|+|\beta|+n+m+a+1)}.
\end{align*}
Next, putting $s=\rho^2(\langle\alpha , \eta\rangle +|\eta|)$,  we obtain
 \begin{align*}
 	I(\alpha,\beta,\gamma)
 	&=(2\pi)^{n+m+1}C_{\alpha,\beta,a}\int_{0}^{\infty}\rho^{2\gamma+1}e^{-\rho^2(\langle\alpha , \eta\rangle +|\eta|)}d\rho\\
 	&=\frac{(2\pi)^{n+m+1}C_{\alpha,\beta,a}}{2(\langle\alpha , \eta\rangle+|\eta|)^{\gamma+1}}\int_{0}^{\infty}s^{\gamma}e^{-s}ds  \\
 	&=\frac{\pi^{n+m+1}\Gamma(a+1)\alpha!\beta!\gamma!}{\Gamma(|\alpha|+|\beta|+n+m+a+1)(\langle\alpha , \eta\rangle+|\eta|)^{\gamma+1}}.
 \end{align*}
Now we can compute the formula of weighted Bergman kernel. For the convenience of writing, we write $n+m+a+1$ as $\kappa$, then according to  the
representation form of reproducing kernel (\ref{kernelform}),
	\begin{align}\label{kernel_ex2}
	&~~K_{\mathcal{A}^2(V_{\eta},\varphi^{a})}((z,z',w),(s,s',t))\\
	=&\sum_{\alpha\in \mathbb N^n}\sum_{\beta\in \mathbb N^m}\sum_{\gamma=0}^{\infty} I^{-1}(\alpha,\beta,\gamma)z^{\alpha}{z'}^{\beta}w^{\gamma}\overline{{s}^{\alpha}{s'}^{\beta}t^{\gamma}}\nonumber\\
	=&\frac{1}{\pi^{n+m+1}a!}
	\sum_{\alpha\in \mathbb N^n}\sum_{\beta\in \mathbb N^m}\sum_{\gamma=0}^{\infty}
	\frac{\Gamma(|\alpha|+|\beta|+\kappa)(\langle\alpha , \eta\rangle+|\eta|)^{\gamma+1}}{\alpha!\beta!\gamma!}z^{\alpha}{z'}^{\beta}w^{\gamma}\overline{{s}^{\alpha}{s'}^{\beta}t^{\gamma}}\nonumber\\
	=&\frac{1}{\pi^{n+m+1}a!}
	\sum_{\alpha\in \mathbb N^n}\sum_{\beta\in \mathbb N^m}\frac{\Gamma(|\alpha|+|\beta|+\kappa)(\langle\alpha ,\eta\rangle+|\eta|)}{\alpha!\beta!}z^{\alpha}{z'}^{\beta}\overline{{s}^{\alpha}{s'}^{\beta}}\sum_{\gamma=0}^{\infty}\frac{(\langle\alpha , \eta\rangle+|\eta|)^{\gamma}}{\gamma!}w^{\gamma}\overline{t}^{\gamma}\nonumber\\
	=&\frac{1}{\pi^{n+m+1}a!}\sum_{\alpha\in \mathbb N^n}\sum_{\beta\in \mathbb N^m}
	\frac{\Gamma(|\alpha|+|\beta|+\kappa)(\langle\alpha , \eta\rangle+|\eta|)e^{(\langle\alpha , \eta\rangle+|\eta|)w\bar{t}}}{\alpha!\beta!}z^{\alpha}{z'}^{\beta}\overline{{s}^{\alpha}{s'}^{\beta}}\nonumber\\
	=&\frac{1}{\pi^{n+m+1}a!}\sum_{\alpha\in \mathbb N^n}
	\frac{(\langle\alpha , \eta\rangle+|\eta|)e^{(\langle\alpha , \eta\rangle+|\eta|)w\bar{t}}}{\alpha!}z^{\alpha}\overline{{s}^{\alpha}}
	\sum_{\beta\in \mathbb N^m}\frac{\Gamma(|\alpha|+|\beta|+\kappa){z'}^{\beta}\overline{s'}^{\beta}}{\beta!}.\nonumber
    \end{align}
Using formula (\ref{equation lemma3 multiindex}), we have
\begin{align*}
	\sum_{\beta\in \mathbb N^m}\frac{\Gamma(|\alpha|+|\beta|+\kappa){z'}^{\beta}\overline{s'}^{\beta}}{\beta!}
	=&\sum_{l=0}^{\infty}\frac{\Gamma(|\alpha|+l+\kappa)}{l!}
	\sum_{|\beta|=l}
	\frac{l!}{\beta!}{z'}^{\beta}\overline{{s'}^{\beta}}\\
	=&\sum_{l=0}^{\infty}\frac{\Gamma(|\alpha|+l+\kappa)}{l!}\langle z', s' \rangle^l.
\end{align*}
Therefore
\begin{align*}
(\ref{kernel_ex2})&=\frac{1}{\pi^{n+m+1}\Gamma(a+1)}
\sum_{l=0}^{\infty}\frac{\langle z', s' \rangle^l}{l!}
\sum_{\alpha\in \mathbb N^n}
\frac{\Gamma(|\alpha|+l+\kappa)(\langle\alpha , \eta\rangle+|\eta|)e^{(\langle\alpha , \eta\rangle+|\eta|)w\bar{t}}}{\alpha!}z^{\alpha}\overline{{s}^{\alpha}},
\end{align*}
where $\kappa=n+m+a+1$.
To calculate $K_{\mathcal{A}^2(V_{\eta},\varphi^{a})}((z,z',w),(s,s',t))$, we split the sum into two pieces:
    \begin{align*}
    \begin{split}    		
	&~~K_{\mathcal{A}^2(V_{\eta},\varphi^{a})}((z,z',w),(s,s',t))\\
	=&\frac{e^{|\eta| w\bar{t}}}{\pi^{n+m+1}\Gamma(a+1)} \left({\sum_{l=0}^{\infty}\frac{\langle z', s' \rangle^l}{l!}\sum_{\alpha\in \mathbb N^n}
	 \frac{\Gamma(|\alpha|+l+\kappa)\langle\alpha , \eta\rangle e^{\langle\alpha , \eta\rangle w\bar{t}}}{\alpha!}z^{\alpha}\overline{{s}^{\alpha}}}\right.\\
	 &\qquad \left.{+\sum_{l=0}^{\infty}\frac{\langle z', s' \rangle^l}{l!}\sum_{\alpha\in \mathbb N^n}
	\frac{\Gamma(|\alpha|+l+\kappa)|\eta| e^{\langle\alpha , \eta\rangle w\bar{t}}}{\alpha!}z^{\alpha}\overline{{s}^{\alpha}} }\right) \\
	=&:\tilde{C}_{\eta,a}(I_2+I_3),
	\end{split}
    \end{align*}
where $\tilde{C}_{\eta,a}=\frac{e^{|\eta| w\bar{t}}}{\pi^{n+m+1}\Gamma(a+1)}$. Sums $I_2$ and $I_3$ remain to be computed. Then, for the first piece, 
	\begin{align*}
I_2
    =&\sum_{l=0}^{\infty}\frac{\langle z', s' \rangle^l}{l!}\sum_{\alpha\in \mathbb N^n}
	\frac{\Gamma(|\alpha|+l+\kappa)\langle\alpha , \eta\rangle e^{\langle\alpha , \eta\rangle w\bar{t}}}{\alpha!}z^{\alpha}\overline{{s}^{\alpha}}\\
	=&\sum_{l=0}^{\infty}\frac{\langle z', s' \rangle^l}{l!}\sum_{k=1}^{\infty}\frac{\Gamma(k+l+\kappa)}{k!}\sum_{|\alpha|=k}
	\frac{k!\langle\alpha , \eta\rangle e^{\langle\alpha , \eta\rangle w\bar{t}}}{\alpha!}z^{\alpha}\overline{{s}^{\alpha}}	\\
	=&\sum_{l=0}^{\infty}\frac{\langle z', s' \rangle^l}{l!}\sum_{k=1}^{\infty}\frac{\Gamma(k+l+\kappa)}{k!}\sum_{j=1}^{n}\eta_jz_j\frac{\partial}{\partial z_j}\left(\sum_{|\alpha|=k}
	\frac{k! e^{\langle\alpha , \eta\rangle w\bar{t}}}{\alpha!}z^{\alpha}\overline{{s}^{\alpha}}	\right) .
    \end{align*}
 We will write $\zeta_j=e^{\eta _jw\bar{t}}z_j\overline{s}_j$ for simplicity. According to (\ref{equation lemma3 multiindex}), we can know that $\sum_{|\alpha|=k}
 \frac{k! e^{\langle\alpha , \eta\rangle w\bar{t}}}{\alpha!}z^{\alpha}\overline{{s}^{\alpha}}=(\zeta_1+\cdots+ \zeta_n)^{k}$, then 
\begin{align}\label{SUM2}	
I_2
	=&\langle\zeta , \eta\rangle
	\sum_{l=0}^{\infty}\frac{\langle z', s' \rangle^l}{l!}\sum_{k=1}^{\infty}\frac{\Gamma(k+l+\kappa)}{(k-1)!}(\zeta_1+\cdots+ \zeta_n)^{k-1}\\
	=&\langle\zeta , \eta\rangle
	\sum_{l=0}^{\infty}\frac{\langle z', s' \rangle^l}{l!} \Gamma(l+\kappa+1)
	\sum_{k=0}^{\infty}\frac{\Gamma(k+l+\kappa+1)}{k!\Gamma(l+\kappa+1)}(\zeta_1+\cdots+ \zeta_n)^{k}\nonumber\\
	=&\langle\zeta , \eta\rangle \sum_{l=0}^{\infty}\frac{\langle z', s' \rangle^l}{l!}\Gamma(l+\kappa+1)\frac{1}{(1-|\zeta|)^{l+\kappa+1}}\nonumber\\
	=&\frac{\langle\zeta , \eta\rangle \Gamma(\kappa+1)}{(1-|\zeta|)^{\kappa+1}} \sum_{l=0}^{\infty}\frac{\Gamma(l+\kappa+1)}{l!\Gamma(\kappa+1)}\left(\frac{\langle z', s' \rangle}{1-|\zeta|}\right)^{l}\nonumber\\
	=&\frac{\langle\zeta , \eta\rangle \Gamma(\kappa+1)}{(1-|\zeta|)^{\kappa+1}} \frac{1}{(1-\frac{\langle z', s' \rangle}{1-|\zeta|})^{\kappa+1}}\nonumber\\
	=&\frac{\langle\zeta , \eta\rangle \Gamma(\kappa+1)}{(1-|\zeta|-\langle z', s' \rangle)^{\kappa+1}},\nonumber
    \end{align}
where $|\zeta|=\zeta_1+\cdots+ \zeta_n=\sum_{j=1}^{n} e^{\eta_{j} w \bar{t}} z_{j} \bar{s}_{j}$. Summation of $I_3$ is straightforward:
\begin{eqnarray}\label{SUM3}
I_3
	&=&\sum_{l=0}^{\infty}\frac{\langle z', s' \rangle^l}{l!}\sum_{\alpha\in \mathbb N^n}
	\frac{\Gamma(|\alpha|+l+\kappa)|\eta| e^{\langle\alpha , \eta\rangle w\bar{t}}}{\alpha!}z^{\alpha}\overline{{s}^{\alpha}} \\
	&=&|\eta|\sum_{l=0}^{\infty}\frac{\langle z', s' \rangle^l}{l!}\sum_{\alpha\in \mathbb N^n}
	\frac{\Gamma(|\alpha|+l+\kappa)}{\alpha!}\zeta^{\alpha}\nonumber\\	
	&=&|\eta|\sum_{l=0}^{\infty}\frac{\langle z', s' \rangle^l}{l!}\sum_{k=0}^{\infty}
	\frac{\Gamma(k+l+\kappa)}{k!}|\zeta|^{k}\nonumber\\
	&=&|\eta|\sum_{l=0}^{\infty}\frac{\langle z', s' \rangle^l}{l!}\frac{\Gamma(l+\kappa)}{(1-|\zeta|)^{l+\kappa}}\nonumber\\
	&=&\frac{|\eta|\Gamma(\kappa)}{(1-|\zeta|)^{\kappa}}\sum_{l=0}^{\infty}\frac{\Gamma(l+\kappa)}{l!\Gamma(\kappa)}\left(\frac{\langle z', s' \rangle}{1-|\zeta|}\right)^{l}\nonumber\\
	&=&\frac{|\eta|\Gamma(\kappa)}{(1-|\zeta|)^{\kappa}}\frac{1}{(1-\frac{\langle z', s' \rangle}{1-|\zeta|})^{\kappa}}
	=\frac{|\eta|\Gamma(\kappa)}{(1-|\zeta|-\langle z', s' \rangle)^{\kappa}}.	\nonumber
\end{eqnarray}
Combining (\ref{SUM2}) with (\ref{SUM3}) and noting $\kappa=n+m+a+1$, we now have
 	\begin{align*}
	&~K_{\mathcal{A}^2(V_{\eta},\varphi^{a})}((z,z',w),(s,s',t))\\
	=&\frac{e^{|\eta| w\bar{t}}}{\pi^{n+m+1}\Gamma(a+1)}\left(\frac{\Gamma(n+m+a+2) \sum_{j=1}^{n} \eta_{j} e^{\eta_{j} w \bar{t}} z_{j} \bar{s}_{j}}{\phi^{n+m+a+2}}+\frac{|\eta|\Gamma(n+m+a+1)}{\phi^{n+m+a+1}}\right),
    \end{align*}
where $\phi\left(z, z', w ; s, s', t\right)=1-\sum_{j=1}^{n} e^{\eta_{j} w \bar{t}} z_{j} \bar{s}_{j}-\langle z', s' \rangle$.
	\end{proof}

%\section*{Acknowledgments}
%The authors sincerely thank the referee for their many valuable and constructive suggestions, which made this paper more readable. 
%The project is supported by the National Natural Science Foundation of China (Grant no. 12071035 and 11971045).

%\subsection*{Author contribution}
%All authors contributed equally to the writing of this paper. All authors read and approved the final manuscript.

%\section*{Declarations}
%\subsection*{Competing interests}
%The authors declare no conflict of interests.

\end{document}